\documentclass{amsart}

\usepackage{psfrag}
\usepackage{color}
\usepackage{tikz}
\usetikzlibrary{matrix,arrows}
\usepackage{graphicx,graphics}
\usepackage{amssymb,amsfonts,amsmath,amstext,amsthm,amscd,verbatim,enumerate}
\usepackage[T1]{fontenc}

\DeclareMathOperator{\dist}{dist}

\begin{document}

\newtheorem{theorem}{Theorem}[section]
\newtheorem{result}[theorem]{Result}
\newtheorem{fact}[theorem]{Fact}
\newtheorem{conjecture}[theorem]{Conjecture}
\newtheorem{lemma}[theorem]{Lemma}
\newtheorem{proposition}[theorem]{Proposition}
\newtheorem{corollary}[theorem]{Corollary}
\newtheorem{facts}[theorem]{Facts}
\newtheorem{props}[theorem]{Properties}
\newtheorem*{thmA}{Theorem A}
\newtheorem{ex}[theorem]{Example}
\theoremstyle{definition}
\newtheorem{definition}[theorem]{Definition}
\newtheorem{remark}[theorem]{Remark}
\newtheorem{example}[theorem]{Example}
\newtheorem*{defna}{Definition}

\newcommand{\notes} {\noindent \textbf{Notes.  }}
\newcommand{\defn} {\noindent \textbf{Definition.  }}
\newcommand{\defns} {\noindent \textbf{Definitions.  }}
\newcommand{\x}{{\bf x}}
\newcommand{\e}{\epsilon}
\renewcommand{\d}{\delta}
\newcommand{\z}{{\bf z}}
\newcommand{\B}{{\bf b}}
\newcommand{\V}{{\bf v}}
\newcommand{\T}{\mathbb{T}}
\newcommand{\Z}{\mathbb{Z}}
\newcommand{\Hp}{\mathbb{H}}
\newcommand{\D}{\Delta}
\newcommand{\R}{\mathbb{R}}
\newcommand{\N}{\mathbb{N}}
\renewcommand{\B}{\mathbb{B}}
\renewcommand{\S}{\mathbb{S}}
\newcommand{\C}{\mathbb{C}}
\newcommand{\ft}{\widetilde{f}}
\newcommand{\dt}{{\mathrm{det }\;}}
 \newcommand{\adj}{{\mathrm{adj}\;}}
 \newcommand{\0}{{\bf O}}
 \newcommand{\av}{\arrowvert}
 \newcommand{\zbar}{\overline{z}}
 \newcommand{\xbar}{\overline{X}}
 \newcommand{\htt}{\widetilde{h}}
\newcommand{\ty}{\mathcal{T}}
\newcommand\diam{\operatorname{diam}}
\renewcommand\Re{\operatorname{Re}}
\renewcommand\Im{\operatorname{Im}}
\newcommand{\tr}{\operatorname{Tr}}
\renewcommand{\skew}{\operatorname{skew}}

\newcommand{\ds}{\displaystyle}
\numberwithin{equation}{section}

\renewcommand{\theenumi}{(\roman{enumi})}
\renewcommand{\labelenumi}{\theenumi}

\newcommand{\note}[1]{\marginpar{\tiny #1}}
\newcommand{\vyron}[1]{{\scriptsize \color{blue}\textbf{Vyron's note:} #1 \color{black}\normalsize}}
\newcommand{\alastair}[1]{{\scriptsize \color{red}\textbf{Alastair's note:} #1 \color{black}\normalsize}}

\title{On uniformly disconnected Julia sets}

\author{Alastair N. Fletcher}
\email{fletcher@math.niu.edu}
\address{Department of Mathematical Sciences, Northern Illinois University, Dekalb, IL 60115, USA}

\author{Vyron Vellis}
\email{vvellis@utk.edu}
\address{Department of Mathematics, The University of Tennessee, Knoxville, TN 37966, USA}

\thanks{The first named author was supported by a grant from the Simons Foundation (\#352034, Alastair Fletcher). The second named author was partially supported by NSF DMS grant 1952510.}

\maketitle

\begin{abstract}
It is well-known that the Julia set of a hyperbolic rational map is quasisymmetrically equivalent to the standard Cantor set. Using the uniformization theorem of David and Semmes, this result comes down to the fact that such a Julia set is both uniformly perfect and uniformly disconnected. We study the analogous question for Julia sets of UQR maps in $\S^n$, for $n\geq 2$. Introducing hyperbolic UQR maps, we show that the Julia set of such a map is uniformly disconnected if it is totally disconnected. Moreover, we show that if $E$ is a compact, uniformly perfect and uniformly disconnected set in $\S^n$, then it is the Julia set of a hyperbolic UQR map $f:\S^N \to \S^N$ where $N=n$ if $n=2$ and $N=n+1$ otherwise.
\end{abstract}

\section{Introduction}

In \cite{DSbook}, David and Semmes introduced  a scale-invariant version of total disconnectedness towards a uniformization of all metric spaces that are quasisymmetric to the standard middle-third Cantor set $\mathcal{C}$: \emph{A set $X\subset \R^n$ is quasisymmetrically homeomorphic to $\mathcal{C}$ if and only if it is compact, uniformly disconnected and uniformly perfect.}

A rich source of Cantor set constructions in $\S^n$, for $n\geq 2$, arises from dynamics. As observed in \cite{HaPi}, if $f:\mathbb{S}^2 \to \mathbb{S}^2$ is a hyperbolic rational map for which the Julia set is totally disconnected, then $J(f)$ is quasisymmetrically equivalent to $\mathcal{C}$. 
Comparing with the uniformization result of David and Semmes, it is clear that $J(f)$ is compact. Moreover, it is well-known that $J(f)$ is uniformly perfect, see for example \cite{MaDR}. Hence the important property here is that for a hyperbolic rational map, if $J(f)$ is totally disconnected, then it is uniformly disconnected. Informally, this means that on all scales, the points of $J(f)$ do not cluster together too much, and is in some sense the opposite notion to uniform perfectness.

The condition that $f$ is hyperbolic cannot be dropped here. Every uniformly disconnected set $X \subset \R^2$ is porous and by a result of Luukainen \cite[Theorem 5.2]{Lu}, the Assouad dimension of $X$, and so also the Hausdorff dimension, is strictly less than $2$. However, Yang \cite{Ya} exhibited cubic polynomials with totally disconnected Julia set and Hausdorff dimension equal to $2$. In these examples, $J(f)$ contains a critical point.

In this paper, we explore the analogous situation in the context of uniformly quasiregular mappings in $\S^n$, for $n\geq 2$. For brevity we will call them UQR maps. This class of mappings is the correct generalization of complex dynamics to higher real dimensions, with a well developed theory. See Bergweiler's survey \cite{Be} and Martin's survey \cite{Martin-survey} for an introduction to the subject. Again it is clear that if $f:\mathbb{S}^n \to \mathbb{S}^n$ is UQR, then $J(f)$ is compact. Moreover, $J(f)$ is uniformly perfect \cite{FN}. So again the question comes down to the property of uniform disconnectedness.

Our first result shows that for hyperbolic UQR maps, totally disconnected implies uniformly disconnected. We make the definition for hyperbolic UQR maps in the preliminaries below, but it is the same as for rational maps: the Julia set must not intersect the post-branch set. This definition is new in the context of UQR maps, but as we note in section \ref{sec:prelims} the class is non-empty.

\begin{theorem}\label{thm:hyperbolic-UD}
Let $n\geq 2$. If $f:\S^n \to \S^n$ is a hyperbolic UQR map and the Julia set is totally disconnected, then it is uniformly disconnected. 
\end{theorem}

Therefore, by the uniformization result of David and Semmes, if $f:\mathbb{S}^n \to \mathbb{S}^n$ is a hyperbolic UQR map, then $J(f)$ is quasisymmetrically equivalent to $\mathcal{C}$. Note, however, that this does not mean that $J(f)$ is ambiently homeomorphic to $\mathcal{C}$ since there do exist hyperbolic UQR maps for which $J(f)$ is a wild Cantor set, see \cite{FlWu}.

The next result addresses the converse question of when a uniformizable totally disconnected subset of $\S^n$ is a Julia set of a hyperbolic UQR map.

\begin{theorem}\label{thm:UD-hyperbolic}
Let $n\geq 2$. If $E\subset \S^n$ is a compact, uniformly perfect and uniformly disconnected set, then it is the Julia set of a hyperbolic UQR map $f:\S^{N} \to \S^{N}$, where $N=2$ if $n=2$ and $N=n+1$ if $n\geq 3$.
\end{theorem}

One of the tools used in the proof of this result is the conformal trap method \cite{IM}. This yields a hyperbolic UQR map $G:\S^N \to \S^N$ with $J(G)$ equal to the standard Cantor set $\mathcal{C}$. Consequently, if $F:\S^N\to \S^N$ is a quasiconformal map, then $F(\mathcal{C})$ is a Cantor set that also arises as a Julia set of a UQR map. This UQR map is just a conjugate of $G$.

This idea gives one way of improving Theorem \ref{thm:UD-hyperbolic}. Following \cite[Definition 1.2]{Fl}, we say that an iterated function system (IFS) $\mathcal{F} = \{ \phi_1,\dots,\phi_n\}$ of contracting similarities has the \emph{strong ball open set condition} if there exists a topological ball $B \subset \R^3$ such that $\phi_1(\overline{B}),\dots,\phi_n(\overline{B})$ are mutually disjoint and contained in $B$. Here by topological ball we mean the image of $\mathbb{B}^3$ under a homeomorphism of $\R^3$.

\begin{theorem}\label{thm:UD-R3} If $X$ is the attractor of an IFS satisfying the strong ball open set condition, then $X$ is the image of $\mathcal{C}$ under a quasiconformal homeomorphism of $\R^3$. In particular, $X$ is the Julia set of a hyperbolic UQR map of $\R^3$.
\end{theorem}

This result extends \cite[Theorem 1.3]{Fl} from two to three dimensions since the only conformal contractions of $\R^3$ are similarities \cite{IM2}. As will be clear from the proof of Theorem \ref{thm:UD-R3}, the key obstruction to increasing this result to higher dimensions is the lack of results approximating orientation-preserving homeomorphisms by orientation-preserving diffeomorphisms.

We remark that Theorem \ref{thm:UD-R3} is not true if  $X$ is only assumed to be tame, uniformly disconnected and uniformly perfect.

\begin{proposition}
\label{prop:qss}
There exists a compact, uniformly perfect and uniformly disconnected set $X\subset \R^3$ such that $X$ is ambiently homeomorphic to $\mathcal{C}$ but not ambiently quasiconformal to $\mathcal{C}$.  
\end{proposition}

In a forthcoming paper, the first named author and Stoertz \cite{FS} show that if the Julia set of a hyperbolic UQR map $f:\S^3 \to \S^3$ is a Cantor set, then it has finite genus, that is, there exists a defining sequence comprised of handlebodies with uniformly bounded genus. Moreover, if there exists a point of the Julia set with local genus $g$, then the set of points with local genus $g$ is dense in the Julia set. However, a quasisymmetric image of $\mathcal{C}$ embedded in $\S^3$ need not have this property, for example, the union of an Antoine's necklace with a tame Cantor set separated by a hyperplane. Consequently, not all quasisymmetric images of $\mathcal{C}$ arise as Julia sets of hyperbolic UQR maps $f:\S^3 \to \S^3$. Further work in this direction could ask for a classification of the geometry of totally disconnected Julia sets for UQR maps which are not hyperbolic, or even if there are non-hyperbolic UQR maps for which the Julia set is totally disconnected. It may be worth pointing out here that while $z\mapsto z^d$ is a hyperbolic rational map, the UQR analogues of these constructed in \cite{Ma} are not hyperbolic. This is because the branch set consists of rays from $0$ to infinity, but the Julia set is the unit sphere in $\R^n$.

The paper is organized as follows. In section \ref{sec:prelims} we recall the basics of UQR maps and introduce hyperbolic UQR maps. We also recall some of the geometric notions we will need. In section \ref{sec:proof1} we prove Theorem \ref{thm:hyperbolic-UD}, in section \ref{sec:proof2} we prove Theorem \ref{thm:UD-hyperbolic} and in section \ref{sec:3} we prove Theorem \ref{thm:UD-R3} and Proposition \ref{prop:qss}.

We would like to thank Peter Haissinsky for helpful comments on the uniform disconnectedness of hyperbolic rational maps. We also thank the anonymous referee for their comments that improved the exposition of the paper.

\section{Preliminaries}
\label{sec:prelims}

We denote by $B(x,r)$ the (open) ball in a metric space $X$ centered at $x\in X$ and of radius $r$. 
For $n\geq 2$ we identify $\R^n \cup \{ \infty \}$ with $\S^n$ and use the chordal metric.
If $X=\S^n$ and we want to emphasize the dimension, we write $B^n(x,r)$.

\subsection{Quasiregular maps}

A continuous map $\Omega: E \to \R^n$ defined on a  domain $\Omega\subset \R^n$ is called \emph{quasiregular}  if $f$ belongs to the Sobolev space $W^{1,n}_{\text{loc}}(E)$ and if there exists some $K\geq 1$ such that 
\[ |f'(x)|^n \leq K J_f(x) \qquad \text{for a.e. $x\in E$}.\]
Here $J_f$ denotes the Jacobian of $f$ at $x\in E$ and $|f'(x)|$ the operator norm. The smallest such $K$ for which this inequality holds is called the outer dilatation and denoted $K_O(f)$. If $f$ is quasiregular, then we also have
\[ J_f(x) \leq K' \min _{|h|=1}  |f'(x)(h)| \qquad \text{ for a.e. $x\in E$}.\]
The smallest $K'$ for which this inequality holds is called the inner dilatation and denoted $K_I(f)$. Then the maximal dilatation of a quasiregular map $f$ is $K(f) = \max \{K_O(f) , K_I(f) \}$. We then say that $f$ is $K(f)$-quasiregular. The \emph{branch set} $\mathcal{B}(f)$ of a quasiregular map $f: E \to \R^n$ is the the closed set of points in $E$ where $f$ does not define a local homeomorphism. See Rickman's monograph \cite{Rickman} for an exposition on quasiregular mappings. 

Quasiregular mappings can be defined at infinity and also take on the value infinity. To do this, if $A:\S^n \to \S^n$ is a M\"obius map with $A(\infty) = 0$, then we require $f\circ A^{-1}$ or $A\circ f$ respectively to be quasiregular via the definition above.

If $f$ is quasiregular and a homeomorphism, then we say that $f$ is \emph{quasiconformal}. Quasiconformality is a generalization of conformality, while quasiregularity is a generalization of holomorphicity.  A notion stronger than that of quasiconformality (and better adapted to a general metric space setting) is that of quasisymmetry. A homeomorphism $f:(X,d) \to (Y,d')$ between metric spaces is \emph{quasisymmetric} if there exists a homeomorphism $\eta: [0,+\infty) \to [0,+\infty)$ such that 
\[ \frac{d'(f(x),f(a))}{d'(f(x),f(b))} \leq \eta\left( \frac{d(x,a)}{d(x,b)}\right) \qquad\text{for all $x,a,b \in X$ with $x\neq b$}.\]
If we want to emphasize the distortion function $\eta$, we say that $f$ is \emph{$\eta$-quasisymmetric}.

\subsection{UQR mappings}

The composition of two quasiregular mappings is always quasiregular but the maximal dilatation typically increases \cite[Theorem II.6.8]{Rickman}. A quasiregular map $f$ is \emph{uniformly quasiregular} (abbv. \emph{UQR}) if there exists $K\geq 1$ such that for every $m\in\N$, the $m$-th iterate $f^{m} = f\circ\cdots\circ f$ is $K$-quasiregular.

If $f:\S^n \to \S^n$ is UQR, then the Fatou set of $f$ is the set
\[ F(f) = \{x\in\S^n : (f^{m}|_U)_{m=1}^{\infty}\text{ is a normal family for some open set $U\ni x$}\}\]
and the Julia set of $f$ is the set $J(f)=\S^n \setminus F(f)$.

In the following proposition, we record some properties of Julia sets of UQR mappings on $\S^n$ that we will need for our proofs. For a map $f:\S^n \to \S^n$ and a point $x\in\S^n$, recall the backward orbit $O^-(x) = \{ y : f^m(y)=x, m\in \N \}$ and the forward orbit $O^+(x) = \{ f^m(x) : m\geq 0 \}$. 

\begin{proposition}
\label{prop:uqr}
Let $n\geq 2$ and suppose that $f:\S^n \to \S^n$ is UQR. Then:
\begin{enumerate}[(i)]
\item $J(f)$ is closed.
\item If $g=f^m$, then $J(g) = J(f)$.
\item $J(f)$ and its complement $F(f)$ are completely invariant under $f$.
\item The exceptional set $\mathcal{E}(f)$ (the set consisting of all points with finite backward orbit) is a finite set. Moreover, if $U$ is any open set intersecting $J(f)$, the forward orbit $O^+(U) = \bigcup_{x\in U}O^+(x)$ contains $\S^n \setminus \mathcal{E}(f)$.
\item For any $x\in \S^n$, the closure of the backward orbit $\overline{O^-(x)}$ contains $J(f)$. If $x\in J(f)$, then it equals $J(f)$.
\item $J(f)$ is uniformly perfect.
\end{enumerate}
\end{proposition}

The proof of the first five of these properties can be found in \cite{Be}. The final property is from \cite{FN}.

We now introduce the notion of a hyperbolic UQR map.

\begin{definition}
Let $n\geq 2$ and let $f:\S^n \to \S^n$ be a non-injective UQR map.
\begin{enumerate}[(i)]
\item The \emph{post-branch set} of $f$ is
\[ \mathcal{P}(f)=\overline{\{ f^m(\mathcal{B}(f) ) : m\geq 0 \}}.\]
\item The map $f$ is called {\it hyperbolic} if $J(f) \cap \mathcal{P}(f)$ is empty.
\end{enumerate}
\end{definition}

This definition is the obvious analogue of the usual one for rational maps, but here it is a little more restrictive since the branch set of a quasiregular map in $\S^n$, for $n\geq 3$, cannot have isolated points. As noted in the introduction, this means that the UQR power maps are not hyperbolic and neither are the UQR analogues of Chebyshev polynomials. However, there do exist hyperbolic UQR maps. The UQR map constructed in \cite{FlWu} is in fact conformal and expanding on a neighbourhood of its Julia set. It follows that the branch set is in the escaping set and hence its orbit cannot approach $J(f)$. Moreover, the conformal trap construction from \cite{Martin97,MP} give hyperbolic UQR maps. Note that all these examples have a totally disconnected Julia sets.

\subsection{Quasi-self-similarity}

A non-degenerate metric space $(X,d)$ is \emph{$c$-uniformly perfect} if there exists $c\geq 1$ such that for any $x\in X$ and any $r \in (0,\diam{X})$, the set $B(x,r)\setminus B(x,r/c)$ is nonempty. A metric space $(X,d)$ is \emph{$c$-uniformly disconnected} if there exists $c\geq 1$ such that for any $r \in (0,\diam{X})$ and any $x\in X$ there exists a set $E \subset X$ containing $x$ such that $\diam{E} \leq r$ and $\dist(E,X\setminus E)\geq r/c$.

Following Carrasco Piaggio \cite{Piaggio}, given a constant $r_0>1$ and a homeomorphism $\eta:[0,+\infty) \to [0,+\infty)$, we say that a metric space $(X,d)$ is $(\eta,r_0)$-\emph{quasi-self-similar} if for every $x\in X$ and $r\in (0,\diam{X})$ there exists an $\eta$-quasisymmetric $\phi: B(x,r) \to X$ such that 
\[ B(\phi(x),r_0) \subset \phi(B(x,r)).\]
Note that our definition of quasi-self-similarity is slightly weaker of that of Carrasco Piaggio as we make no assumption on the size of the ball $B(\phi(x),r_0)$. However, if $X$ is $c$-uniformly perfect, then it is easy to see that $\diam{B(\phi(x),r_0)}\geq r_0/c$. By Proposition \ref{prop:uqr} (vi), we can use this definition of quasi-self-similarity when discussing Julia sets of UQR maps.

\section{Proof of Theorem \ref{thm:hyperbolic-UD}}
\label{sec:proof1}

The aim of this section is prove Theorem \ref{thm:hyperbolic-UD}. Firstly we show that for uniformly perfect and totally disconnected sets, quasi-self-similarity implies uniform disconnectedness. Then we show that Julia sets of UQR maps are quasi-self-similar.

\begin{lemma}\label{lem:QSimpliesUD}
Suppose that $X$ is compact, uniformly perfect, quasi-self-similar and totally disconnected. Then $X$ is uniformly disconnected.
\end{lemma}

\begin{proof}
Suppose that $X$ is $c$-uniformly perfect and $(r_0,\eta)$-quasi-self-similar. Since $X$ is non-degenerate, by rescaling its metric, we may assume that $\diam{X}=1$. Since $X$ is compact, perfect and totally disconnected, there exists a homeomorphism $f:\mathcal{C} \to X$ where $\mathcal{C}$ is the standard Cantor set. Recall that $\mathcal{C}$ is the attractor of the IFS $(\R,\{\phi_1,\phi_2\})$ where
\[ \phi_i(x) = x/3 +2(i-1)/3, \qquad i=1,2.\]

For each $k\in\N$ and $w = i_1\cdots i_k \in \{1,2\}^k$, we set $X_w = f(\phi_{i_1}\circ\cdots\circ\phi_{i_k}(\mathcal{C}))$. 
By the uniform continuity of $f$, there exists $k_0\in\N$ such that for any $w\in \{1,2\}^{k_0}$, 
\[ \diam{X_w} < \d_0:=  r_0 \min\left \{ (2c)^{-1}, (2c)^{-1}\theta^{-1}\left((4c)^{-3}\right) \right\},\]
where $\theta:[0,+\infty) \to [0,+\infty)$ is defined by $\theta(t) = (\eta^{-1}(t^{-1}))^{-1}$. Recall that the inverse of an $\eta$-quasisymmetric map is $\theta$-quasisymmetric \cite[Proposition 10.6]{Heinonen}. Define also
\[ d_0 := \min_{w \in \{1,2\}^{k_0}}\dist(X_w,X \setminus X_w).\]

Fix $x\in X$ and $r>0$. Then, there exists an $\eta$-quasisymmetric map $\phi:B(x,r) \to X$ such that
\[ B(\phi(x),r_0) \subset \phi(B(x,r)).\]
Let $w\in \{1,2\}^{k_0}$ such that $\phi(x)\in X_w$. Then, by the choice of $\d_0$ we have that $X_w \subset B(\phi(x),(2c)^{-1}r_0)$. Set $E = \phi^{-1}(X_w)$. We show that $\diam{E}$ is less than $r$, while its distance from $X\setminus E$ is at least comparable to $r$.

Firstly, by the uniform perfectness of $X$, we know that 
\[ \diam{\phi(B(x,r))} \geq \diam{B(\phi(x),r_0)} \geq c^{-1}r_0.\]
Therefore, by Proposition 10.8 in \cite{Heinonen} and the choice of $\d_0$,
\begin{equation}\label{eq:diam}
\diam{E} \leq \theta\left( 2\frac{\diam{X_w}}{\diam{\phi(B(x,r))}} \right) \diam{B(x,r)} \leq 2\theta\left( 2c\d_0r_0^{-1}\right)r < (2c)^{-3}r.
\end{equation}
By the uniform perfectness of $X$, there exist a point $y_1 \in B(x,r)\setminus B(x,r/c)$ and a point $y_2 \in B(x,2^{-3}c^{-2}r)\setminus B(x,(2c)^{-3}r)$. Therefore,
\[ \diam(B(x,r) \setminus E) \geq |y_1-y_2| \geq  r(c^{-1}-\tfrac18 c^{-3}).\]

We now estimate $\dist(E,X\setminus E)$. By the choice of $\d_0$, we have that $X_w\subset B(\phi(x),(2c)^{-1}r_0)$ and, by uniform perfectness of $X$, $\diam{B(\phi(x),r_0)} \geq r_0/c$. Hence,
\[ \diam(\phi(B(x,r)\setminus E)) \geq \diam(B(\phi(x),r_0)\setminus X_w) \geq (2c)^{-1}r_0.\]
Now, by \cite[p. 532]{Tyson98}, setting $\psi(t) = (\theta(t^{-1}))^{-1}$, we have
\begin{align*} 
\dist(E,X\setminus E) &= \dist(E, B(x,r)\setminus E)\\
&\geq \frac12\psi\left(\frac{\dist(X_w,\phi(B(x,r)\setminus E))}{\diam(\phi(B(x,r)\setminus E))}  \right)\diam(B(x,r)\setminus E)\\
&\geq \frac12\psi\left(\frac{d_0}{\diam{X}}  \right)c^{-1}r\\
&\geq \frac12\psi\left(d_0  \right)c^{-1}r. \qedhere
\end{align*}
\end{proof}

\begin{remark}
For any $C>1$ there exists a $4$-uniformly perfect and $(\eta,1)$-quasi-self-similar set with $\eta(t) =t$, that is not $C$-uniformly disconnected; therefore the uniform disconnectedness constant in Lemma \ref{lem:QSimpliesUD} does not depend only on $r_0$, $\eta$ and $c$, but also on the set. To see this, fix $C>1$ and $\e \in (0, (2C+1)^{-1})$. Let $X$ be the Cantor set which is the attractor of the IFS $(\R,\{\phi_1,\phi_2\})$ with
\[ \phi_i(x) = (1-\e)x/2 + (i-1)(1+\e)/2, \qquad i=1,2.\]
Since $\e< 1/2$, it is easy to see that $X$ is $4$-uniformly perfect. Moreover, since $X$ is self-similar, it is also $(\eta,1)$-quasi-self-similar with $\eta(t) = t$. Now, if $x = 0$, and $r= (1-\e)/2$, then for any $E \subset X\cap B(x,r)$ we have $\dist(E, X\setminus E) \leq \e < C^{-1}r$. Hence, $X$ is not $C$-uniformly disconnected.
\end{remark}

For the rest of this section we will use the chordal metric $\sigma$ on $\S^n$. If $E,F$ are closed sets in $\R^n \cup \{\infty \}$, then $\sigma (E,F)$ denotes the chordal distance between them.
Moreover, given $f: \S^n \to \S^n$, denote by $L_f(x,r)$ the quantity
\[ L_f(x,r) = \max_{\sigma(y,x) = r} \sigma( f(y) , f(x)) .\]

\begin{lemma}
\label{lem:1}
Let $f: \S^n \to \S^n$ be a hyperbolic UQR map. There exists $r_1 > 0$ such that if $x\in J(f)$, then $f$ is injective on $B(x,r_1)$.
\end{lemma}

\begin{proof}
For each $x\in J(f)$, let $r_x$ denote the supremum of radii $r$ so that $f$ is injective on $B(x,r)$. Since $f$ is hyperbolic, $r_x > 0$ for each $x\in J(f)$. The $r_1$ we will obtain is the Lebesgue covering number of the cover $\{ B(x,r_x) : x\in J(f) \}$ of $J(f)$.

Now suppose the result was false. Then there would exist a sequence $x_n \in J(f)$ with $r_{x_n} \to 0$. Passing to a subsequence if necessary, and recalling that $J(f)$ is compact, we may assume by relabelling that $x_n \to x_0$. Since $J(f)$ is closed, $x_0 \in J(f)$. Then there is no neighbourhood of $x_0$ on which $f$ is injective. To see this, if $\epsilon >0$, we can find $N$ large enough so that $B(x_N , r_{x_N}) \subset B(x_0,\epsilon /2)$.

This means that $x_0 \in \mathcal{B}(f)$. However, since $f$ is hyperbolic, we arrive at a contradiction.
\end{proof}

\begin{lemma}
\label{lem:1a}
Let $f: \S^n \to \S^n$ be a non-injective hyperbolic UQR map and let $J(f)$ be a Cantor set. There exists $\epsilon_0 >0$ so that if $0<\epsilon<\epsilon_0$ and $U$ is an $\epsilon$-neighbourhood of $J(f)$ then there exists $N\in N$ such that $\overline{f^{-N}(U)} \subset U$.
\end{lemma}

\begin{proof}
First, by the classification of stable Fatou components given in \cite[Proposition 4.9]{HMM04}, each stable Fatou component is either a (super-)attracting basin, a parabolic basin, or a rotation domain. Since $J(f)$ is a Cantor set, there is just one Fatou component. Further, since $f$ is non-injective, it follows that $F(f)$ is either a (super-)attracting basin containing one fixed point, or a parabolic basin containing no fixed points.

Next, let $S \subset F(f)$ denote the union of the set of fixed points of $f$ in the Fatou set together with the exceptional set $\mathcal{E}(f)$. By the previous paragraph, this is a finite set. Let $\epsilon _0 < \sigma ( J(f) , S)$. Now given $\epsilon \in (0,\epsilon_0)$, let $U$ be the $\epsilon$-neighbourhood of $J(f)$. If the lemma is not true, then for each $k\in \N$, there exist $w_k\in \S^n \setminus U$ and $z_k \in U$ with $f^k(w_k) = z_k$.

Since $\S^n$ is compact, we may pass to subsequences $w_{k_j}$ and $z_{k_j}$ which converge to $w_0 \in \S^n \setminus U$ and $z_0\in \overline{U}$ respectively. By construction, $w_0 \in F(f)$ and so there exists a neighbourhood of $w_0$, say $V$, on which $(f^{k_j}|_U)$ converges uniformly to a fixed point of $f$ in the Fatou set in the (super-)attracting case, or to a fixed point of $f$ in the Julia set in the parabolic case. This fixed point is necessarily $z_0$.

In the case where $z_0 \in F(f)$, we obtain a contradiction because $\overline{U}$ cannot contain any points of $S$. In the case where $z_0 \in J(f)$, since the branch set is contained in $F(f)$ and the subsequence $f^{k_j}$ converges uniformly on compact subsets of $F(f)$ to $z_0$, we obtain a contradiction to $f$ being hyperbolic.
This completes the proof.
\end{proof}

\begin{theorem}
\label{thm:1}
If $f: \S^n \to \S^n$ is a hyperbolic UQR map, then $J(f)$ is quasi-self-similar.
\end{theorem}

\begin{proof}
Recalling $r_1$ from Lemma \ref{lem:1} and $\epsilon _0$ from Lemma \ref{lem:1a}, let 
\[ r_2 < \min \{r_1 ,\epsilon _0, \sigma(J(f) , \mathcal{P}(f)) \}.\] 
Then let $U$ be an $r_2$-neighbourhood of $J(f)$. Note that $U$ cannot be all of $\S^n$ since $\mathcal{B}(f)$ is non-empty.
By construction, $\partial U \subset F(f)$.
By Lemma \ref{lem:1a}, we can find $N\in \N$ such that $\overline{f^{-N}(U)} \subset U$. 

Set $g=f^N$. Then $J(g) = J(f)$ by Proposition \ref{prop:uqr}(ii) and we have $\overline{g^{-1}(U)} \subset U$. In particular, $\partial g^{-1}(U)$ is contained in $U$, is compact and is in $F(f)$. Hence $\sigma( \partial g^{-1}(U) , J(g) ) := \delta > 0$. Moreover, $g^{-1}(U) \cap \mathcal{B}(g) = \emptyset$ since $U \cap \mathcal{P}(f) = \emptyset$.  The point is that if $x\in J(g)$ and $0<t<\delta$, then $g$ is quasiconformal on $B(x,t)$ and $g(B(x,t)) \subset U$. 

Now suppose $r<\delta /2 $ and $x\in J(g)$. Let $B = B(x,r)$ and let $B' = B(x,2r) \subset U$.
Since the forward orbit of $B'$ covers everything except the exceptional set, we can find $M\in \N$ minimal so that 
\begin{equation}
\label{eq:thm1}
L_{g^M}(x,2r) \geq \delta. 
\end{equation}
Then necessarily we have $L_{g^M}(x,2r) \leq r_2$ since it will take at least one more iterate of $g$ for some points in the image of $B'$ to leave $U$. Since $g^i(B') \subset U$ for $i=1,\ldots, M$, it follows that $g^M$ is injective on $B'$ and, in particular, it is quasiconformal on $B'$. The egg-yolk principle \cite[Theorem 11.14]{Heinonen} implies that $g^M$ is $\eta$-quasisymmetric on $B$.
It follows from this and \eqref{eq:thm1} that $g^M(B)$ contains the ball
\[ B\left ( g^M(x) ,\frac{ \delta}{\eta(2)} \right ).\]
We therefore have obtained the condition for quasi-self-similarity of $J(f)$ with $r_0 = \delta/\eta(2)$ and $\phi = g^M|_{J(g)} = f^{NM}|_{J(f)}$. If $r\geq \delta/\eta(2)$, then we may just take $\phi$ to be the identity map. Combining these cases, we conclude that $J(f)$ is quasi-self-similar.
\end{proof}

\section{Proof of Theorem \ref{thm:UD-hyperbolic}}
\label{sec:proof2}

Recall that David and Semmes proved that a set $X\subset\R^n$ is quasisymmetrically homeomorphic to $\mathcal{C}$ if and only if it is compact, doubling, uniformly disconnected and uniformly perfect. Later, MacManus improved that result for sets in $\R^2$ by showing that a set $E \subset\R^2$ is quasisymmetrically homeomorphic to $\mathcal{C}$ if and only if it is the image of $\mathcal{C}$ under a quasiconformal homeomorphism of $\R^2$. (Here and in what follows, for each $n\geq 2$ we identify $\mathcal{C}$ with $\mathcal{C}\times \{(0,\dots,0)\} \subset\R^n$.) MacManus' result is false in $\R^3$ due to the existence of self-similar wild Cantor sets in $\R^3$ \cite[pp. 70--75]{Daverman}, but by increasing the dimension by $1$, MacManus' result generalizes to dimensions $n\geq 3$. See \cite{Semmes1} for related results.

\begin{theorem}[{\cite{MM2,Vext}}]\label{thm:MM}
Given $c,C>1$ and integer $n\geq 2$, there exists $K\geq 1$ depending on $c,C,n$ such that
if a set $E\subset \R^n$ is compact, $c$-uniformly perfect, and $C$-uniformly disconnected, then there exists a $K$-quasiconformal mapping $F:\R^N \to \R^N$ with $F(\mathcal{C}) = E$, where $N=2$ if $n=2$, and $N=n+1$ if $n\geq 3$.
\end{theorem}

For the proof of Theorem \ref{thm:UD-hyperbolic}, we require the following well-known lemma which says that the standard Cantor set is the Julia set of a hyperbolic UQR map. We include a proof for completeness; see also \cite{MP} and \cite{IM}. The main novelty is that we check the constructed map is hyperbolic.

\begin{lemma}\label{lem:Cantorset}
Let $n\geq 2$.
There exists a hyperbolic UQR map $G:\S^n \to \S^n$ whose Julia set is the standard Cantor set $\mathcal{C}$.
\end{lemma}

\begin{proof}
Let $p_0=(-1,0,0,\dots,0)$, $p_1=(0,1,0,\dots,0)$ and $p_2 = (0,-1,0,\dots,0)$. Let $g:\R^n \to \R^n$ with  
\[ g(r,\theta,x_3,\dots,x_n) = (r,2\theta,x_3,\dots,x_n)\]
where the first two coordinates of $\R^n$ are in polar coordinates. It is easy to see that $g$ is a bounded length distortion map with branch set the hyperplane $\{(0,0)\}\times\R^{n-2}$ and that
\[ g^{-1}(p_0) = \{p_1,p_2\}, \quad g(p_0) = (1,0,\dots,0).\]

Let $r_0>0$ so that $g^{-1}(B(p_0,r_0))$ has exactly two components, one containing $p_1$ and another containing $p_2$. Choose also positive constants $a,b$ so that $b<a/2$ and
\begin{enumerate}
\item $B(p_i,a) \subset g^{-1}(B(p_0,r_0))$ for $i=1,2$;
\item $B(p_0,b) \subset g(B(p_i,a))$ for $i=1,2$;
\item $g(B(p_0,b)) \subset B(g(p_0),a) \subset g(B(p_0,r_0))$.
\end{enumerate} 

Now we define $\tilde{g}: \R^n \to \R^n $ with the following rules
\begin{enumerate}
\item $\tilde{g}|_{\R^n \setminus \bigcup_{i=0,1,2}B(p_i,a)} = g|_{\R^n \setminus \bigcup_{i=0,1,2}B(p_i,a)}$;
\item for each $i=0,1,2$, $\tilde{g}|_{B(p_i,b)}$ is a translation of $B(p_i,b)$ onto $B(g(p_i),b)$;
\item on each annulus $B(p_i,a)\setminus B(p_i,b)$, $\tilde{g}$ is defined as the quasiconformal extension of $\tilde{g} : \partial B(p_i,a)\cup \partial B(p_i,b) \to g(\partial B(p_i,a)) \cup\partial B(g(p_i),b)$ given by Sullivan's Annulus Theorem \cite[Theorem 3.17]{TV81}.
\end{enumerate}

Clearly $\tilde{g}$ extends to a quasiregular map $\S^n \to \S^n$ that, by slight abuse of notation, we still call $\tilde{g}$.
Finally, define $G: \S^n  \to \S^n $ by $G = \Phi\circ \tilde{g}$ where $\Phi: \S^n \to \S^n$ is the conformal inversion that maps $\partial B(p_0,b)$ onto itself.

By construction, $f|_{B(p_0,b)}$ is conformal and hence if an orbit ever ends up in $B(p_0,b)$ it stays there. This is called a conformal trap. It turns out that the only way an orbit does not end up in $B(p_0,b)$ is if it stays in $B(p_1,b) \cup B(p_2,b)$. However, $f$ is also conformal on this set. Hence any orbit is obtained by 
\begin{enumerate}
\item either always applying a conformal map, 
\item or applying finitely many conformal maps, then a map with distortion and then conformal maps from there on.
\end{enumerate}

It follows that $G$ is UQR, the Julia set of $G$ is a tame Cantor set contained in $B(p_1,b)\cup B(p_2,b)$ (see \cite{MP}) and that $\mathcal{B}(G) = \mathcal{B}(g) = (\{(0,0)\}\times\R^{n-2}) \cup \{ \infty \}$. Finally, if $x \in \mathcal{B}(G)$, then $\tilde{g}(x)= x$ and $G(x) \in B(p_0,b)$. On the other hand, for any $x\in B(p_0,b)$, we have $G(x) \in B(p_0,b)$. Therefore, $$\mathcal{P}(G) \subset B(p_0,b) \cup (\{(0,0)\}\times\R^{n-2}) \cup \{ \infty \}$$ and $G$ is hyperbolic.
\end{proof}

\begin{proof}[{Proof of Theorem \ref{thm:UD-hyperbolic}}]
Let $F: \R^{N} \to \R^N$ be the quasiconformal map from Theorem \ref{thm:MM}. Clearly $F$ extends to a quasiconformal map $\S^N \to \S^N$ that, again by slight abuse of notation, we still call $F$.

By Lemma \ref{lem:Cantorset}, there exists a non-injective UQR map $G:\S^{N} \to \S^N$ such that $J(G) = \mathcal{C}$ and $G^{-1}(\infty) = \infty$. Define now $f:\R^N \to\R^N$ with $f = F\circ G \circ F^{-1}$. Since $f^{k} = F\circ G^{k} \circ F^{-1}$, it is clear that $f$ is non-injective and UQR. It is immediate that $J(f) = F(J(G)) = F( \mathcal{C}) = E$.

Moreover, $\mathcal{B}(f) = F(\mathcal{B}(G))$ and it follows that $\mathcal{P}(f) = F(\mathcal{P}(G))$. Therefore, since $\mathcal{P}(G)\cap \mathcal{B}(G) = \emptyset$, it follows that $\mathcal{P}(f)\cap \mathcal{B}(f) = \emptyset$ and $f$ is hyperbolic. 
\end{proof}

\section{Self-similar tame Cantor sets in dimension three}
\label{sec:3}

In this section, we discuss when self-similar tame Cantor sets in $\R^3$ are ambiently quasiconformal to $\mathcal{C}$, or not, as the case may be.

\begin{proof}[Proof of Theorem \ref{thm:UD-R3}]
Suppose that $X$ is the attractor of an IFS $\{\phi_1,\dots,\phi_n\}$ satisfying the strong ball open set condition. Let $\mathcal{C}_n$ be the Cantor set which is the attractor of the IFS $\{\psi_1,\dots,\psi_n\}$ where
\[ \psi_i(x,y,z) = \frac{1}{2n-1} (x+ 2i-2, y, z).\]
For $w=i_1\cdots i_k \in\{1,\dots,n\}^k$ we denote $\phi_w = \phi_{i_1}\circ\cdots\circ \phi_{i_k}$ and $\psi_w = \psi_{i_1}\circ\cdots\circ \psi_{i_k}$.

By the David-Semmes uniformization theorem, $\mathcal{C}_n$ is quasisymmetrically homeomorphic to $\mathcal{C}$ and by \cite[Theorem 1.1]{Vext} and Ahlfors extension \cite{Ah64}, there exists a quasiconformal homeomorphism of $\R^3$ that maps $\mathcal{C}_n$ onto $\mathcal{C}$. Therefore, to finish the proof, we construct a quasiconformal homeomorphism $F:\R^3 \to \R^3$ such that $F(X) = \mathcal{C}_n$. Let $B'$ be the ball with center $(\frac12,0,0)$ and radius $5/6$. Then $\mathcal{C}_n \subset B'$ and the balls $\psi_1(\overline{B'}),\dots, \psi_n(\overline{B'})$ are mutually disjoint and all contained contained in $B'$.

Let $B \subset \R^3$ be the image of $\mathbb{B}^3$ under a homeomorphism of $\R^3$ such that $\phi_{1}(\overline{B}), \dots, \phi_{n}(\overline{B})$ are mutually disjoint and contained in $B$. Since orientation preserving homeomorphisms of $\R^3$ can be approximated by orientation preserving diffeomorphisms \cite{Mun60}, there exists a topological ball $B'' \subset B$ with smooth boundary such that the Hausdorff distance between $\partial B$ and $\partial B''$ satisfies
\[ \dist_H(\partial B, \partial B'') < \min\{\dist(\partial \phi_i(B),\partial B) : i=1,\dots,n\}.\]
Then, the balls $\phi_i(\overline{B})$ are all contained in $B''$, so the balls $\phi_1(\overline{B''}),\dots,\phi_n(\overline{B''})$ are disjoint and are contained in $B''$. Note that the set
\[ K= \bigcap_{k=1}^{\infty} \bigcup_{w \in \{1,\dots,n\}^k}\phi_{w} (\overline{B''})\]
is compact and invariant under the IFS $\{\phi_1,\dots,\phi_n\}$ so by the uniqueness of the attractor \cite{Hutchinson}, $K=X$; so $X \subset B''$.

Define $f:\partial B''\cup \bigcup_{i=1}^n\phi_i(\partial B') \to \partial B'\cup \bigcup_{i=1}^n\psi_i(\partial B')$ so that $f|_{\partial B''} : \partial B'' \to \partial B'$ is an orientation preserving diffeomorphism and for each $i=1,\dots,n$, 
\[ f|_{\phi_i(\partial B'')} = \psi_i \circ f|_{\partial B''}\circ\phi_{i}^{-1}|_{\phi_i(\partial B'')}.\] 
We claim that there exists a quasiconformal extension 
\[ F : \overline{B''} \setminus \bigcup_{i=1}^n\phi_i(B'') \to \overline{B'} \setminus \bigcup_{i=1}^n\psi_i(B').\] 
Assuming the claim, we can extend $F$ quasiconformally to $\overline{B''} \setminus X$ by setting 
\[ F|_{\phi_w(\overline{B''}) \setminus \bigcup_{i=1}^n\phi_{wi}(B'')} = \psi_w\circ F_{\overline{B''} \setminus \bigcup_{i=1}^n\phi_{i}(B'')} \circ \phi_w^{-1}, \quad\text{for $w\in \{1,\dots,n\}^k$}.\]
Moreover, we can extend $F$ quasiconformally to $\R^3 \setminus B''$ by Ahlfors extension theorem \cite{Ah64}. Now, by a theorem of V\"ais\"al\"a for removable singularities \cite[Theorem 35.1]{Vais71}, $F$ extends quasiconformally to $\R^3$ and maps $X$ onto $\mathcal{C}_n$.

To prove the claim, let $Q''$, $\D''$, $Q'$, $\D'$, $Q_1'',\dots,Q_n''$, $\D_1'',\dots,\D_n''$, $Q_1',\dots,Q_n'$, $\D_1',\dots,\D_n'$ be open cubes in $\R^3$ with the following properties:
\begin{enumerate}
\item $\overline{B''} \cup \overline{\D''} \subset Q''$ and $\overline{B'} \cup\overline{\D'} \subset Q'$;
\item for each $i\in\{1,\dots,n\}$ we have 
\[ \overline{Q_i''} \subset \phi_i(B'') \cap \D_i'' \subset \overline{\D_i''} \subset \D'' \quad\text{and}\quad \overline{Q_i'} \subset B_i' \cap \D_i' \subset \overline{\D_i'} \subset \D'.\] 
\end{enumerate}
We now construct two quasiconformal maps
\[ G: \overline{B''}\setminus \bigcup_{i=1}^n \phi_i(B'') \to \overline{Q''}\setminus \bigcup_{i=1}^n Q_i'', \qquad G': \overline{B'}\setminus \bigcup_{i=1}^n \psi(B') \to \overline{Q'}\setminus \bigcup_{i=1}^n Q_i'.\]
Assuming we have these maps, we set 
\[ h: \partial Q''\cup \bigcup_{i=1}^n \partial Q_i'' \to \partial Q' \cup\bigcup_{i=1}^n \partial Q_i' \qquad\text{with } h = G'\circ f \circ G^{-1}.\]
Applying Sullivan's Annulus Theorem, we can extend $h$ to
\[ h: (\overline{Q''}\setminus \D'')\cup \bigcup_{i=1}^n (\overline{\D_i''}\setminus Q_i'') \to (\overline{Q'}\setminus \Delta')\cup \bigcup_{i=1}^n (\overline{\D_i'}\setminus Q_i')\]
so that $h|_{\partial \D''}$ is a similarity mapping $\partial\D''$ onto $\partial \D'$ and for each $i\in\{1,\dots,n\}$, $h|_{\partial \D_i''}$ is a similarity mapping $\partial\D_i''$ onto $\partial \D'_i$. By Proposition 4.8 in \cite{Vext}, there exists a quasiconformal extension of $h$
\[ H : \overline{Q''}\setminus \bigcup_{i=1}^n Q_i'' \to \overline{Q'}\setminus \bigcup_{i=1}^n Q_i'\]
and we can set $F = (G')^{-1}\circ H \circ G$.

It remains to construct the maps $G,G'$. We only work for $G$; the construction of $G'$ is similar. Let $D'', D_1'',\dots,D_n''$ be balls with smooth boundary such that $\overline{D''} \subset B''$, and for every $i=1,\dots,n$, $\phi_i(\overline{B}'') \subset D_i'' \subset \overline{D_i''} \subset D''$. Define now $G$ as follows: 
\begin{enumerate}
\item $G|_{\partial B''}: \partial B'' \to \partial Q''$ is an orientation preserving diffeomorphism, $G|_{\partial D''}: \partial D'' \to \partial D''$ is the identity and $G|_{\overline{B''}\setminus D''} : \overline{B''}\setminus D'' \to \overline{Q''}\setminus B''$ is the quasiconformal extension of the latter two diffeomorphisms given by Sullivan's Annulus Theorem;
\item $G|_{\phi_i(\partial B'')}: \phi_i(\partial B'') \to \partial Q_i''$ is an orientation preserving diffeomorphism, $G|_{\partial D_i''}: \partial D_i'' \to \partial D_i''$ is the identity and $G|_{\overline{D_i''}\setminus \phi_i(B'')} : \overline{D_i''}\setminus \phi_i(B'') \to \overline{D_i''}\setminus Q_i''$ is the quasiconformal extension of the latter two diffeomorphisms given by Sullivan's Annulus Theorem;
\item $G|_{\overline{D''}\setminus \bigcup_{i=1}^nD_i''} :\overline{D''}\setminus \bigcup_{i=1}^n D_i''\to \overline{D''}\setminus \bigcup_{i=1}^n D_i''$ is the identity. \qedhere
\end{enumerate} 
\end{proof}

\begin{proof}[Proof of Proposition \ref{prop:qss}]
Let $\mathcal{F} = \{\phi_i : \R^3 \to \R^3\}_{i=1}^{n}$ be the IFS generating the Antoine necklace. In particular, there exists a closed solid torus $T \subset \mathbb{B}^3$ such that the tori $\phi_i(T)$ are mutually disjoint, are contained in the interior of $T$, have the same diameter and form a chain inside $T$ with $\phi_i(T)$ linked with $\phi_{i+1}(T)$ for all $i\in\{1,\dots,n-1\}$, and $\phi_n(T)$ linked with $\phi_1(T)$; see \cite[\textsection 3.1]{FlWu} for a precise description. Let also $\{\phi_i': \R^3 \to \R^3\}_{i=1}^{n}$ be contracting similarities such that the closed balls $\phi_i'(\overline{\mathbb{B}^3})$ are mutually disjoint, have equal diameters and are contained in the interior of $T$.

For each finite word $w$ made up of letters in $\{1,\dots,n\}$ we construct a similarity $\psi_w$ in an inductive manner. For each $i\in\{1,\dots,n\}$, define $\psi_i = \phi_i'$. Inductively, assume that for some $k\in\N$ and some word $w$ in $\{1,\dots,n\}^{k}$ we have defined $\psi_w$. 
\begin{itemize}
\item If $k+1 =2^m$ for some $m\in\N$, then for any $i \in \{1,\dots,n\}$ set $\psi_{wi} = \psi_w \circ \phi_{i}'$.
\item If $k+1 \neq 2^m$ for some $m\in\N$, then for any $i \in \{1,\dots,n\}$ set $\psi_{wi} = \psi_w \circ \phi_{i}$.
\end{itemize}
Let $X$ be the Cantor set defined as
\[ X = \bigcap_{k=1}^{\infty}\bigcup_{w\in\{1,\dots,n\}^k}\psi_w(T)\]
It is straightforward to check that $X$ is compact, uniformly perfect and uniformly disconnected. Moreover, 
\[ X = \bigcap_{m =1}^{\infty} \bigcup_{w \in \{1,\dots,n\}^{2^m}}  \psi_w(\overline{\mathbb{B}^3})\]
with balls $\{\psi_w(\overline{\mathbb{B}^3})\}_{w \in \{1,\dots,n\}^{2^m}}$ being mutually disjoint. Thus, $X$ is tame. 

For a contradiction, assume that there exists a $K$-quasiconformal map $f:\R^3 \to \R^3$ such that $f(X) = \mathcal{C}\subset \R\times \{(0,0)\}$. Let $k=2^m$, let $w\in \{1,\dots,n\}^k$, and let $A_1=\psi_{w1}(X)$ and $A_2=\psi_{w2}(X)$. Recall that by our construction, $\psi_{w1}(T)$ is linked with $\psi_{w2}(T)$. Note that 
\[ \diam{A_1} = \diam{A_2} = C_0\dist(A_1,A_2).\]
for some universal $C_0>0$. By the quasisymmetry of $f$, there exists $C>1$ depending only on $K$ such that 
\[ C^{-1}\diam{f(A_1)} \leq \dist(f(A_1),f(A_2)) \leq C\diam{f(A_1)}\]
\[C^{-1}\diam{f(A_1)} \leq \diam{f(A_2)} \leq C\diam{f(A_1)}.\]

Note that both $f(A_1)$ and $f(A_2)$ are contained in the line $\R\times\{(0,0)\}$. For each $i=1,2$, fix $x_i \in f(A_i)$ and let $E_i$ be the union of all line segments joining the point $(x_i, (-1)^i\diam{A_i},0)$ with the elements of $f(A_i)$. For each $i=1,2$, let $B_i$ by the $\e$-neighborhood of $E_i$ with $\e = \frac1{4}\dist(f(A_1),f(A_2))$. Then, $B_1$ and $B_2$ are topological balls in $\R^3$ that contain $f(A_1)$ and $f(A_2)$, respectively, such that for $i=1,2$ and for all $x\in f(A_i)$
\[ (C^{-1}/4)\diam{f(A_1)} \leq \dist(x,\partial B_i) \leq (C/4)\diam{f(A_1)}.\]

By the quasisymmetry of $f^{-1}$, there exists $C'>1$ depending only on $K$, and there exist two mutually disjoint topological balls $B_1' = f^{-1}(B_1)$ and $B_2' = f^{-1}(B_2)$ such that for $i=1,2$, $A_i \subset B_i'$, and for all $x\in A_i$
\[ (C')^{-1}\diam{A_1} \leq \dist(x,\partial B_i') \leq C'\diam{A_1}.\]

However, assuming that $k$ is sufficiently large, there exists $l \in \{ k+2, \dots, 2k -1\}$ such that for all words $u \in\{1,\dots,n\}^{l-k-1}$ each $i\in\{1,2\}$ and all $x\in \partial \psi_{wiu}(T)$
\[ \dist(x,A_i) < (C')^{-1}\diam{A_1}.\]
Define now 
\[ M_1 = \bigcup_{u \in\{1,\dots,n\}^{l-k-1}} \psi_{w1u}(T) \qquad\text{and}\qquad M_2 = \bigcup_{u \in\{1,\dots,n\}^{l-k-1}} \psi_{w2u}(T).\]
On the one hand, by construction, $M_1$ is linked with $M_2$ in $\R^3$. On the other hand, $M_1 \subset B_1'$ and $M_2\subset B_2'$ so they are unlinked in $\R^3$; a contradiction.
\end{proof}

\bibliography{UD-bib}
\bibliographystyle{amsbeta}

\end{document}